\documentclass[a4paper, 11 pt, twoside]{amsart}
\usepackage[left = 3.5cm, right = 3.5cm, headsep = 6mm,
footskip = 10mm, top = 35mm, bottom = 35mm, footnotesep=5mm, headheight =
2cm]{geometry}
\usepackage{graphicx}
\usepackage[expansion=false]{microtype}
\usepackage{booktabs, multirow,  tabularx}

\usepackage{amsmath}
 \usepackage{amscd}
 \usepackage{amssymb}

\usepackage[usenames, dvipsnames]{xcolor}

\usepackage{tikz, tikz-cd}
\usepackage{enumitem}
\setlist[description]{leftmargin=0cm,  labelindent=\parindent}
\usepackage{verbatim}

\usepackage[unicode,bookmarks, pdftex, backref = page]{hyperref}
\hypersetup{colorlinks=true,citecolor=NavyBlue,linkcolor=NavyBlue,urlcolor=Orange, pdfpagemode=UseNone, breaklinks=true}

\newtheoremstyle{special-example}
  {}
  {}
  {}
  {\parindent}
  {\bfseries}
  {:}
  { }
  {}
  \theoremstyle{special-example}
\newtheorem{example}[equation]{Example}

\renewcommand{\tilde}{\widetilde}

\newcommand{\OO}{\mathcal O}

\newcommand{\wt}{\widetilde}

\newcommand{\inv}{^{-1}}
\newcommand{\inff}{_{\infty}}
\newcommand{\margincom}[1]{\marginpar{\scriptsize\textsl{ #1}}}
\newcommand{\I}{\mathrm{i}}

\newcommand{\Q}{\mathbb Q}

\newcommand{\F}{\mathbb F}

\newcommand{\pp}{\mathbb P}

\newcommand{\cQ}{\mathcal Q}

\DeclareMathOperator{\pic}{Pic}

\newcommand{\fie}{\varphi}

\numberwithin{equation}{section}

\newcommand{\sr}[1]{\textsf{\scriptsize\color{Plum} S: {#1}}}

\newcommand{\mf}[1]{\textsf{\scriptsize\color{Red} M: {#1}}}
\newcommand{\rp}[1]{\textsf{\scriptsize\color{Orange} R: {#1}}}

\newcommand{\mfr}[1]{\margincom{\mf{#1}}}

\newcommand{\jr}[1]{\textsf{\scriptsize\color{Green} J: {#1}}}

\newtheoremstyle{Lehn-it}
  {}
  {}
  {\itshape}
  {}
  {\bfseries}
  {$\;$\textmd{---}}
  { }
  {}

\newtheoremstyle{Lehn-up}
  {}
  {}
  {\upshape}
  {}
  {\bfseries}
  {$\;$\textmd{---}}
  { }
  {}

  \newtheoremstyle{up-list}
  {}
  {}
  {\upshape}
  {}
  {\bfseries}
  { }
  { }
  {}

\newtheoremstyle{Lehn-Bemerkung}
  {}
  {}
  {}
  {}
  {\itshape}
  {$\;$\textmd{---}}
  { }
  {}

\newtheoremstyle{citing}
  {}
  {}
  {\itshape}
  {}
  {\bfseries}
  {$\;$\textmd{---}}
  {.5em}
  {\thmnote{#3}}
  
\numberwithin{equation}{section}

{\theoremstyle{Lehn-it}

\newtheorem{thm}[equation]{Theorem}
\newtheorem{lem}[equation]{Lemma}
\newtheorem{prop}[equation]{Proposition}
\newtheorem{cor}[equation]{Corollary}

}
{\theoremstyle{Lehn-up}

}
{\theoremstyle{Lehn-Bemerkung}
\newtheorem{rem}[equation]{Remark}
}

{\theoremstyle{up-list}

}

{\theoremstyle{citing}
}

\newcommand{\onto}{\twoheadrightarrow}
\newcommand{\into}{\hookrightarrow}
\newcommand{\ol}{\overline}
\newcommand{\sm}{_{\rm sm}}
\newcommand{\gm}{_{\rm Gmin}}

\DeclareFontFamily{OT1}{rsfs}{}
\DeclareFontShape{OT1}{rsfs}{n}{it}{<-> rsfs10}{}
\DeclareMathAlphabet{\curly}{OT1}{rsfs}{n}{it}

\DeclareMathOperator{\Aut}{Aut}

\DeclareMathOperator{\rk}{rk}

\newcommand{\isom}{\cong}

\renewcommand{\epsilon}{\varepsilon}
\renewcommand{\phi}{\varphi}
\renewcommand{\theta}{\vartheta}

\newcommand{\kb}{{\mathcal B}}

\newcommand{\ki}{{\mathcal I}}

\newcommand{\ko}{{\mathcal O}}

\newcommand{\kq}{{\mathcal Q}}

\newcommand{\kx}{{\mathcal X}}

\newcommand{\IA}{{\mathbb A}}

\newcommand{\IC}{{\mathbb C}}

\newcommand{\IF}{{\mathbb F}}

\newcommand{\IN}{{\mathbb N}}

\newcommand{\IP}{{\mathbb P}}
\newcommand{\IQ}{{\mathbb Q}}

\newcommand{\gothM}{{\mathfrak M}}

\title{  I-surfaces with one T-singularity}

\author{Marco Franciosi}
\address{Marco Franciosi\\Dipartimento di Matematica\\Universit\`a di Pisa \\Largo B. Pontecorvo 5\\I-56127  Pisa\\Italy}
\email{marco.franciosi@unipi.it}

\author{Rita Pardini}
\address{Rita Pardini\\Dipartimento di Matematica\\Universit\`a di Pisa \\Largo B. Pontecorvo 5\\I-56127  Pisa\\Italy}
\email{rita.pardini@unipi.it}

\author{Julie Rana}
\address{Julie Rana\\Department of Mathematics, Lawrence University, 711 E. Boldt Way, Appleton WI 54911, USA.}
\email{julie.f.rana@lawrence.edu}

\author{S\"onke Rollenske}
\address{S\"onke Rollenske\\FB 12/Mathematik und Informatik\\
Philipps-Universit\"at Marburg\\
Hans-Meerwein-Str. 6\\
35032 Marburg\\
Germany}
\email{rollenske@mathematik.uni-marburg.de}

\dedicatory{This article is dedicated  to Fabrizio Catanese,\\ who showed  the beauty of algebraic surfaces to us and many others.}
\begin{document}
\begin{abstract}
 We classify normal stable surfaces with $K_X^2 = 1$, $p_g = 2$ and $q=0$ with a unique singular point which is a non-canonical T-singularity,  thus exhibiting two divisors in the main component and a new irreducible component of the moduli space of stable surfaces $\overline\gothM_{1,3}$. 
\end{abstract}

\maketitle
\setcounter{tocdepth}{1}
\tableofcontents
\section{Introduction} 

The investigation of (minimal) surfaces of general type with low invariants and their moduli spaces  started with the work of Castelnuovo and Enriques (cf. \cite{ Enriques-book})
 at the beginning of the 20th century and has remained an active topic ever since. It is fair to say, that Fabrizio Catanese, to whom the present special  issue is dedicated, was one the most influential contributors to this topic over the last decades both through his work, e.g., 
\cite{MR555688,MR963062,MR1689180,MR2400886,MR2964466,MR3184164,MR4176836},  and by passing on his enthusiasm to his students and collaborators. 
Nowadays Gieseker's moduli space of canonical models $\gothM_{K^2, \chi}$ \cite{gieseker77} is known to admit a modular compactification $\overline \gothM_{K^2, \chi}$, the KSBA moduli space of stable surfaces, see Section \ref{sec: stable-surf}.

In this article, we continue the investigation of (the moduli space of) stable I-surfaces\footnote{The name was coined by Green, Griffiths, Laza, and Robles in their investigation of Hodge-theoretic stratifications of the moduli space.} begun in \cite{FPR17}. These are stable surfaces with $K_X^2 = 1$, $p_g = 2$, and $q=0$. The Gieseker moduli space $\gothM_{1,3}\subset \overline\gothM_{1,3}$ is an irreducible and rational variety of dimension $28$ parametrising double covers of the quadric cone $\cQ_2\subset \IP^3$ branched over a quintic section and the vertex, a fact that was attributed to Kodaira in  \cite[\S 3]{Horikawa2} and extended to the moduli space $\overline\gothM_{1,3}^{(Gor)}$  of Gorenstein stable surface in \cite{FPR17}.

Another important class of singularities on stable surfaces are T-singularities, which are exactly the quotient singularities that can occur in  stable surfaces in the closure of the main component (see Section \ref{sec: T-sing}). 
Here we study what we call T-singular I-surfaces, that is, stable I-surfaces with a unique singular point which is a non-canonical  T-singularity.  Our main result is:

\begin{thm}\label{thm: oneT}
 Let $X$ be an I-surface with unique singular point a non-canonical  T-singularity. Then only the following cases can occur:
 \begin{center}
   \begin{tabular}{cll}
  \toprule
Cartier index &   T-singularity& Construction \\
  \midrule
   $2$ &  $ \frac{1}{4d}(1,2d-1)$ $(d\leq 32)$ & Example \ref{ex: n=2}\\
   $3$ &  $ \frac{1}{18}(1,5)$&  Example \ref{ex: (b)}\\
   $5$ &  $\frac{1}{25}(1, 14)$ & Example \ref{ex: (c)}\\
  \bottomrule
 \end{tabular}
 \end{center}
 \end{thm}
 
 If all deformations are unobstructed, then the dimension of the $\IQ$-Gorenstein deformation space of the T-singularity gives the codimension of the corresponding stratum in the moduli space. 
 This is especially interesting for Wahl singularities, which are expected to give rise to divisors. This expectation is only partially met (see  Proposition \ref{prop: D}) in our cases.
 For T-singular I-surfaces we obtain the following
\begin{cor}\label{cor: moduli}
  The T-singular I-surfaces of type $\frac 14(1,1)$ and $\frac {1} {18}(1,5)$ form divisors in the main component of $\overline \gothM_{1,3}$, that is, the closure of the Gieseker moduli space. The T-singular surfaces  of type $\frac 1{25}(1,14)$ form an open subset of another irreducible component of $\overline \gothM_{1,3}$ of dimension $28$.
 \end{cor}
Schematically one might depict the situation as follows, where we include a conjectural connection between the two components:

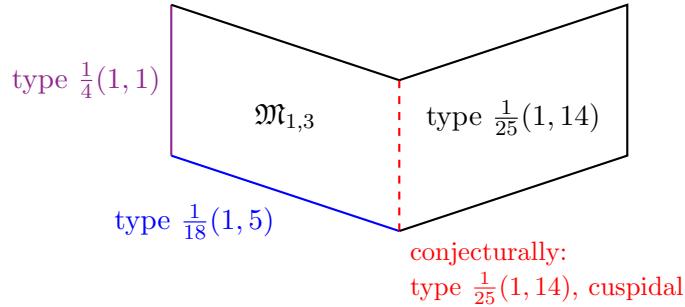
\begin{center}
 \begin{tikzpicture}[thick]
 \draw[Plum] (0,2) to node [left] {type $\frac 14(1,1)$}  (0,0);
 \draw[blue] (0,0) to node [below left] {type $\frac {1}{ 18}(1,5)$ } (3,-1);
  \draw (0,2) to (3,1);
  \draw[red, dashed] (3,1) to (3,-1) node[below right, text width = 4cm] {
  {\small conjecturally:\\ type $\frac 1 {25}(1,14)$, cuspidal}}; 
  \node at (1.5, 0.5) {$\gothM_{1,3}$};
  \draw (3,1) -- (6,2) -- (6,0) -- (3,-1);
  \node at (4.5, 0.5) {type $\frac 1 {25}(1,14)$};

 \end{tikzpicture}

\end{center}

\subsection*{Acknowledgements}
M.F., R.P. and S.R. would like to thank Fabrizio Catanese, who shaped their view on mathematics in general and algebraic surfaces in particular. 
We are indebted to Valery Alexeev for explaining Proposition \ref{prop: aut-ext} to us. 
The first and second author are  partially supported by the project PRIN 
 2017SSNZAW$\_$004 ``Moduli Theory and Birational Classification"  of Italian MIUR. 
 The first and second author  are members of GNSAGA of INDAM.

\hfill\break
{\bf Notation and conventions:} We work over the complex numbers. $\sim$ denotes linear equivalence.  For $n$  a positive integer, a  $(-n)$-curve $C$ on a smooth surface $Y$ is a smooth rational curve with $C^2=-n$.
 For a Hirzebruch surface $\F_n$,  we denote by $\sigma\inff$ the infinity section and by $\Gamma$ the class of a ruling, so that  a section $\sigma_0$ disjoint from $\sigma\inff$ is linearly equivalent to
  $n\Gamma +\sigma\inff$; we denote by $\kq_n$ the cone in $\pp^{n+1}$ over the rational normal curve of degree $n$ in $\pp^n$, which is the image of $\IF_n$ via the map given by $|\sigma_0|$.
\section{Preliminaries }\label{sec: normal surfaces}

\subsection{Normal stable surfaces}\label{sec: stable-surf}
Stable surfaces were first defined to give a geometric compactification of the moduli space of surfaces of general type (see \cite{ksb88, kollar12, kovacs15} and references therein).
In the construction of the moduli space, one of the main insights was that one cannot allow all flat families of stable surfaces, but only so-called $\IQ$-Gorenstein deformations, that is, flat families $\pi \colon \kx\to B$ of stable surfaces with fixed invariants such that for all $m$ the reflexive powers $\omega_{\kx}^{[m]}$ are flat over $B$ and commute with base change. 

Here we will  consider normal stable surfaces only. 
Recall that a  normal  surface $X$ is called  \emph{stable} if  it has  log-canonical singularities, and  $K_X$ is $\Q$-Cartier and ample.  The smallest  $m>0$ such that $mK_X$ is Cartier is called the \emph{(Cartier) index}  of $X$. 

For a normal stable surface $X$, 
we let $f\colon \wt Y\to X$ be the minimal desingularization and    $\eta\colon \wt Y\to Y$ be the morphism to a minimal model:   \[\begin{tikzcd}
{} & \wt Y  \arrow{dr}{f}\arrow{dl}[swap]{\eta}\\ Y && X  
 \end{tikzcd}
\]
\begin{rem}\label{rem: pg constant}
If $X$ has only rational singularities (for instance cyclic quotient singularities),  the Leray spectral sequence gives   
 $q(Y)=q(\wt Y)=q(X)$ and $p_g(Y)=p_g(\wt Y)=p_g(X)$. 
  The normal stable surfaces  considered in this paper have rational singularities and $p_g(X)=2$, so the Kodaira dimension of $\tilde Y$ is  $\ge 1$ and the minimal model $Y$ is unique. 
\end{rem}

Let $X$ be a normal  stable  surface  with  log-terminal singularities.  It is possible to generalize the plurigenus formula for smooth minimal surfaces  as follows.
 Let $\tilde Y \to X$ be the minimal desingularization and write 
\[f^*K_X=K_{\tilde Y}+\Delta = K_{\tilde Y} + \sum_i a_i E_i\]
where $-a_i$ is the log discrepancy of the exceptional curve $E_i$.
Then by Prop.~5.2 and Thm.~5.3 in  \cite{bla95a}, one has:
\begin{equation}\label{eq: plurigenus-lt}
h^0(mK_X)=\chi(\OO_X)+\frac{m(m-1)}{2}K^2_X+\frac 12\{m\Delta\}\left(\{m\Delta\}-\{\Delta\}\right),
\end{equation}
where $\{D\}$, as customary,  denotes the fractional part of a $\Q$-divisor $D$.

\subsection{T-singularities and T-singular surfaces}\label{sec: T-sing}

A T-singularity is either a rational double point or a   quotient 2-dimensional singularity of type $\frac{1}{dn^2}(1, dna-1)$, where $n>1$ and  $d,a >0$ are integers with $a$ and $n$ are coprime.
These are precisely the quotient singularities that admit a  $\Q$-Gorenstein smoothing, that is, that can occur on smoothable stable surfaces (cf.  \cite[\S~3]{ksb88}).

The exceptional divisor of the minimal  resolution of  a T-singularity $\frac{1}{dn^2}(1, dna-1)$ is a so-called \emph{T-string}, a string of rational curves $A_1, A_2, \ldots , A_r$ with self-intersec\-tions
 $-b_1,-b_2, \ldots, -b_r$ given by the Hirzebruch-Jung continued fraction expansion $[b_1,b_2,\ldots, b_r]$ of $\frac{dn^2}{dna-1}$ (see, e.g.,~\cite[Chapter 10]{cls2012}). 
 Following popular convention, we will refer to the expansion $[b_1,b_2,\ldots, b_r]$ corresponding to a T-singularity as a T-string. 
 
The index $2$ T-singularity with $d=1$ has T-string $[4]$. Those of index $2$ and $d>1$ have T-string $[3,2,\ldots, 2, 3]$, where $2$ occurs $d-2$ times. 
 It is immediate to check that all the log discrepancies are equal to $-\frac 12$.
More generally, if $[b_1,\ldots, b_r]$ is the T-string of a T-singularity $\frac{1}{dn^2}(1, dna-1)$ for some $n>2$, then $b_1=2$ (or $b_r=2$) and $[b_2,\ldots, b_r-1]$ 
(respectively, $[b_1-1,\ldots, b_{r-1}]$) is also the T-string of a T-singularity of type $\frac{1}{dn'^2}(1, dn'a-1)$ for some $n'< n$.
In particular, we obtain all possible T-strings of T-singularities of fixed $d$ by beginning with the corresponding T-string of index $2$ listed above and iterating as described~\cite[\S~3]{ksb88}. 
The T-singularities of index 3 are obtained by a single iteration.

A \emph{T-singular surface of type  $\frac{1}{dn^2}(1, dna-1)$} is a normal surface with a singular point of  type  $\frac{1}{dn^2}(1, dna-1)$ and  smooth elsewhere. The index of $X$ is equal to $n$. Using the notation of  \S \ref{sec: stable-surf}  we have (\cite[Prop. 20]{lee98}): 
\begin{equation}\label{eq: lee formula}
K^2_{\tilde Y}=K^2_X-(r-d+1).
\end{equation}

If, in addition, $Y$ is not a rational surface, then by  \cite[Prop. 2.3]{rana-urzua17}: 
\begin{equation}\label{eq: K2}
K^2_Y<K^2_X
\end{equation}

\begin{rem}\label{rem: T-surf}
Let $\tilde Y$ be a smooth surface containing a T-string.   Then by \cite[Prop.~4.10]{Kollar-Mori} there is a map $f\colon \tilde Y\to X$ that contracts the T-string to a T-singularity;  the surface $X$ is projective and $\Q$-factorial, so it is a T-singular surface.  If the class of  $f^*K_X$ is nef and big and the only curves with zero intersection with it are the components of the T-string, then $X$ is also stable. 
\end{rem}

\subsection{I-surfaces}\label{ssection: I-surf} 

An \emph{ I-surface}  $X$  is a stable surface with   $K_X^2=1$, $p_g(X)=2$ and $q(X)=0$. 
In \cite{FPR17}  it was shown  that the classical description of smooth surfaces of general type with
$K_X^2=1$ and $\chi(X)=3$  extends to  the Gorenstein case, i.e.: 

\begin{itemize}
\item  a Gorenstein I-surface $X$  is canonically embedded as a hypersurface of degree $10$ in (the smooth locus of) $\IP(1,1,2,5)$;
 
 \item the moduli space 
$\overline\gothM_{1,3}^{(Gor)}$ of Gorenstein stable surfaces with $K^2=1$ and $\chi=3$ is irreducible and rational of dimension $28$.
\item  for a Gorenstein I-surface $X$, the bicanonical map is a degree 2 morphism $\phi_2\colon X\to \kq_2 \subset \IP^3$, where $\kq_2$  is  the quadric cone, branched on the vertex $O$   and on a quintic section $D$ of $\kq_2$ not containing $O$
\item  conversely, if $D$ is a quintic section of $\kq_2$ not containing $O$ and $(\kq_2, \frac 12 D)$ is a log-canonical pair, then the double cover of $\kq_2$ branched on $D$ and  $O$ is a stable Gorenstein I-surface.   
\end{itemize}
\subsection{ Extending automorphisms of stable surfaces}

Let $X$ be a stable surface and let $g\colon \mathcal X\to B $  be a $1$-parameter $\IQ$-Gorenstein smoothing of $X$; denote by $0\in B$ the point such that $g\inv(0)=X$,  write $B^*:=B\setminus\{0\}$ and $\mathcal X^*:=\mathcal X_{| B^*}$. 

 Let   $\Aut(\mathcal X/B)$  be the relative automorphism scheme. The following result is certainly well known to experts; we thank  V.\ Alexeev for explaining it to us. Note as a starting point, that the automorphim group of a stable surface is finite by \cite[Thm.~11.12]{Iitaka82}, or more generally by \cite{HMX13}.

\begin{prop}\label{prop: aut-ext} Let $\sigma$ be a section of $\Aut(\mathcal X^*/B^*)$. 
Then, up to a finite base change,  $\sigma$ extends to a section of $\Aut(\mathcal X/B)$.
\end{prop}
\begin{proof}
The claim  follows   from the fact that the  family $\mathcal X $  is the canonical model of any extension of $\mathcal X ^*$, and the canonical model is unique.
Indeed, choose an  extension  $\mathcal X' $ of $\mathcal X ^*$ such that $\sigma$ induces a morphism $\mathcal X '\to \mathcal X $; up to a base change we may assume that both $\kx$ and $\kx'$ admit a semi-stable resolution.  Now taking  canonical models of  both  $\mathcal X '$ and $\mathcal X $ one gets a regular map $\bar \sigma\colon \mathcal X \to \mathcal X $ that restricts to $\sigma$ on $\mathcal X ^*$. Since $\sigma^m$ is the identity for some $m$, we have that  $\bar \sigma^m$ is also the identity,  and therefore $\bar \sigma$ is an automorphism.

%
\end{proof}

As a result, we obtain the following necessary condition for smoothability of I-surfaces:
\begin{cor}\label{cor: smoothable-invo} Let $X$ be  a stable  I-surface. If $X$ has a $\IQ$-Gorenstein smoothing  then it admits an involution.
\end{cor}
\begin{proof}  Let $\mathcal X\to B$ be a $1$-parameter $\Q$-Gorenstein smoothing of $X$. By  \cite[Prop. 3.6]{FPR17} (see \S \ref{ssection: I-surf})  the bicanonical map of a Gorenstein I-surface is of degree 2,
 hence  the corresponding involution defines a section of  $\Aut(\mathcal X^*/B^*)$, that, possibly up to a base change,  extends to a section $\sigma$ of  $\Aut(\mathcal X/B)$ by  
Proposition  \ref{prop: aut-ext}.

 The map $\Aut(\mathcal X/B)\to B$ is quasi-finite and \'etale (cf.  \cite[Thm.~3.29]{Alexeev96}),
 so the restriction of $\sigma$ to the central fibre $X$  has indeed order 2. \end{proof}

\section{The examples}
\subsection{The case of index $n=2$}
We start by describing a construction  of  stable T-singular I-surfaces  of index 2. 
\begin{example}\label{ex: n=2}
Let  $\cQ_2\subset \pp^3$ be the quadric cone and let  $\epsilon\colon \F_2\to\cQ_2$ be  the minimal resolution;  as usual, we denote by $\sigma\inff$ the infinity section of $\F_2$, by $\Gamma$ the class of a ruling and write $\sigma_0=\sigma_{\infty}+2\Gamma$. Let $D\subset \F_2$ be an effective divisor linearly equivalent to $4\sigma_0+2\Gamma$ and assume  that $D$ does not contain $\sigma\inff$ and is smooth away from $\sigma\inff$, so that $D$ is either smooth or has a double point $P$  on $\sigma\inff$. 

We let $\pi\colon Y\to  \F_2$ be the double cover branched on $D$.  The surface   $Y$ is smooth when $D$ is,  and    has a singular point $Q$  of type $A_k$ lying over  $P\in \sigma\inff$ otherwise. The linear system $\pi^*|\Gamma|$   is a pencil of  elliptic curves and coincides with the canonical system of $Y$.  We denote by $ \tilde Y\to Y$ the minimal desingularization;  since the singularities of $Y$ are canonical,  $\tilde Y$ is minimal elliptic with $p_g(\tilde Y)=2$ and again the canonical system coincides with the elliptic pencil. 
 There are the following possibilities for the preimage $C$ of $\sigma\inff$ in $\tilde Y$:
 \begin{enumerate}
 \item[(1)] if $D$ meets $\sigma\inff$ at two distinct points, then $\tilde Y=Y$ and $C$ is a $(-4)$-curve;
 \item[(2)] if $D$ is smooth but meets $\sigma\inff$ at only one point, then $\tilde Y=Y$ and  $C$ is a string of type $[3,3]$;
 \item[(3)] if $D$ has a double point $P\in \sigma\inff$ then the point $Q\in Y$ lying over $P$ is an $A_k$ point for some $k>0$. The preimage of  $\sigma\inff$  in $Y$ splits as $C_1+C_2$, with $C_1$ and $C_2$ smooth rational curves meeting at $Q$,  and  $C$   is a string of type $[3,2,\dots,2,  3]$ with 2 occurring $k$ times.
 \end{enumerate}
 
  Let $\nu\colon \tilde Y\to X$ be the first step in the Stein factorization of the map $\tilde Y\to Y \to \cQ_2$: 
   $\nu$ contracts $C$ to a point $R$ lying over the vertex of $\cQ_2$  and is an isomorphism elsewhere.
    So $R$ is a singularity of type $\frac{1}{4d}(1, 2d-1)$ (see \S \ref{sec: T-sing}):  
    in  case (1) above one has  $d=1$, in  case (2) one has   $d=2$, and in   case (3)  one has  $d=k+2>2$. 
    So we have $C^2=-4$, $K_{\tilde Y}C=2$ and  $\nu^*K_X=K_{\tilde Y}+\frac 12 C$  (see \S \ref{sec: T-sing}) and therefore
     $K_X^2=(K_{\tilde Y}+\frac 12 C)^2=1$. Finally, it is easy to check that $\nu^*K_X$ is nef and that the only irreducible curves
      $A$  with $\nu^*K_XA=0$ are the components of $C$, and  therefore  $K_X$ is ample. 
      Finally, $p_g(X)=p_g(\tilde Y)=2$ by Remark \ref{rem: pg constant}, since T-singularities are rational. 
      Summing up, $X$ is a T-singular I-surface of type $\frac{1}{4d}(2d-1)$, where $d=1$ in case (1), $d=2$  in case (2),  and $d=k+2>2$ in case (3).
\end{example}

\begin{rem}\label{rem: bic-n=2}
Let $X$ be an I-surface as in Example \ref{ex: n=2}.  Then the  induced map $\epsilon\colon X\to \cQ_2$ is a finite double cover, flat away from the vertex of $\cQ_2$,
 with branch locus $\bar D= \epsilon(D)$ cut out on $\cQ_2$ by a quintic hypersurface passing through the vertex of $\cQ_2$. 
 By the Hurwitz formula we have $2K_X=\epsilon^*H$, where $H$ is the class of a hyperplane section of $\cQ_2\subset \pp^3$. 
 Since the canonical system $|2K_X|$ is 3-dimensional by  \eqref{eq: plurigenus-lt}, it coincides with $\epsilon^*|H|$. So the  map $\epsilon\colon X\to \cQ_2$ is the bicanonical map of $X$. In particular,  the branch divisor $D$ is determined by $X$ up to automorphisms of $\cQ_2$.

 In fact, our conditions on $D$ guarantee that $(\cQ_2, \frac 12 \bar D)$ is a log-terminal pair and thus $X$ is an I-surface by \cite[Prop.~4.1]{FPR17}. Deforming $\bar D$ to a general quintic section of $\cQ_2$ gives a smoothing of $X$ as a hypersurface of degree $10$ inside $\IP(1,1,2,5)$.

We record this fact for later reference.
  \end{rem}
\begin{cor}\label{cor: n=2-smoothable}
The I-surfaces constructed in Example \ref{ex: n=2} are smoothable.
\end{cor}

To construct surfaces as in Example \ref{ex: n=2}  one needs to find a branch divisor $D$ with the required singularity at a point $P\in\sigma\inff$ and smooth everywhere else. 
This is a non trivial question, as shown  by the following partial  answer to it.  
\begin{prop} \label{prop: D}
Consider T-singular I-surfaces of type $\frac {1}{4d} (1,2d-1)$ as constructed in Example \ref{ex: n=2}. Then
\begin{enumerate}
\item[(i)] we  have  $d\le 32$;
\item [(ii)] for $d=1, 2, 3$ the T-singular I-surfaces of type $\frac {1}{4d} (1,2d-1)$ obtained  as in Example  \ref{ex: n=2} give an irreducible  subvariety of codimension $d$  inside the main component of the moduli space of stable I-surfaces. In particular, for $d=1$ one has a divisor. 
\item[(iii)]   There are irreducible families of  T-singular I-surfaces of type  $\frac{1}{4d}(1, 2d-1)$  depending on $\mu$ moduli for the following values of $(d,\mu)$:
$$ (9,19),\quad (21,7),\quad  (25,4).$$ 
\end{enumerate}
\end{prop}
\begin{proof} We use freely the notation of Example  \ref{ex: n=2}.
\medskip

(i) Let $X$ be as in Example  \ref{ex: n=2} with a singularity of type  $\frac {1}{4d} (1,2d-1)$, with $d>2$. Then  the branch locus $D\subset \F_2$ of the corresponding double cover has exactly a double point $P\in\sigma_{\infty}$ of type $A_{d-2}$ and is smooth elsewhere. 

Assume first that  $D$ is irreducible and let $\tilde D\to D$ be the normalization. Since $D\sim 4\sigma_0+2\Gamma$, we have $p_a(D)=15$ and $p_a(\tilde D)=p_a(D)-\lfloor \frac{d-1}{2}\rfloor$, and therefore  $d\le 32$. 

Assume now that $D=D_1+D_2$ is reducible. If $D_1$ does not intersect $\sigma_\infty$, then $D_1$ and $D_2$ have to intersect away from $\sigma_\infty$, contradicting our assumption that $D$ is smooth away from the section at infinity. Thus $D_1$ and $D_2$ are smooth divisors  not containing $\sigma\inff$ meeting only at the point $P\in \sigma\inff$. 
Write $m=D_1D_2$:  the singular point $P$ is of type $A_{2m-1}$, so we have $d=2m+1$. Up to exchanging  $D_1$ and $D_2$,   there are exactly the following possibilities:
\begin{itemize}
\item[(R1)] $D_1\sim \Gamma$, $D_2\sim 4\sigma_0+\Gamma$, $m=4$, $d=9$;
\item[(R2)] $D_1\sim \sigma_0+ \Gamma$, $D_2\sim 3\sigma_0+\Gamma$, $m=10$, $d=21$;
\item[(R3)] $D_1\sim 2\sigma_0+ \Gamma$, $D_2\sim 2\sigma_0+\Gamma$, $m=12$, $d=25$.
\end{itemize}
\medskip

(ii) The surface $X$ is determined by the choice of $D\in |4\sigma_0+2\Gamma|$ up to the action of the automorphism group of $\F_2$ (equivalently, of $\cQ_2$), which has dimension 7.   The case $d=1$ corresponds to a general choice of $D$, so the number of moduli is $\dim |4\sigma_0+2\Gamma|-7=27$. 
The case $d=2$  correspond to $D$ smooth  but tangent to $\sigma\inff$ at a point and the case $d=3$ corresponds to $D$ with an ordinary double point on $\sigma\inff$. In both cases simple arguments based on Bertini's theorem show that $D$ can be chosen  in an irreducible and  locally closed subset of codimension $d-1$ of  $ |4\sigma_0+2\Gamma|$.
\medskip

(iii) The three families correspond to the cases (R1), (R2) and (R3) above. 

We discuss   case (R3) first.  
Let $D_1\in|2\sigma_0+\Gamma|$ a smooth curve such that the point $P:=D\cap \sigma\inff$ is a Weierstrass point of $D_1$. By Lemma \ref{lem: g=2} there is a $3$-dimensional family of such curves, up to the action of the automorphisms of $\F_2$. 

Consider  the following exact sequence:
\[ 0\to \OO_{\F_2}\to \OO_{\F_2}(D_1)\to \OO_{D_1}(12P)\to 0.\]
Passing to cohomology, we have a surjection $H^0 (\OO_{\F_{2}}(D_1))\onto    H^0(\OO_{D_1}(12P))$; so there is a curve $D'_1\in|2\sigma_0+\Gamma|$  that meets $D_1$ only at $P$. If we take a general element $D_2$ of the pencil spanned by $D_1$ and $D'_1$ we obtain an example  of case (R3). The curve $D_1$ depends on three moduli up to automorphisms of $\F_2$ and for each choice of $D_1$ we have a $1$-dimensional family of possible $D_2$, so case (R3) gives  an irreducible subvariety of dimension 4 of the main component of the  moduli space of I-surfaces.

Consider now case (R2). Because   $h^1(\IF_2,2\sigma_0)=0$, we have a short exact sequence:
\[0\to H^0(2\sigma_0)\to H^0(3\sigma_0+\Gamma)\to H^0(\OO_{D_1}(10P))\to 0.\]
So the curves in $|3\sigma_0+\Gamma|$ that cut out the divisor $10P$ on $D_1$ are a linear subsystem $|M|$ of dimension $9=h^0(2\sigma_0)$; $P$ is the only base point of $|M|$  and for general $R\in|2\sigma_0|$  the curve  $D_1+R$ is smooth at $P$. So by Bertini's Theorem, we can pick $D_2$ in a non empty open subset of $|M|$, and we in total  have $h^0(\sigma_0+\Gamma)-1+9 =14$ parameters for the construction. Taking into account the action of the automorphism group of $\F_2$, which is $7$-dimensional,  we see that we have  7 moduli.

Case (R1) can be analysed in the same way: one gets $25+1=26$ parameters for the construction and therefore 19 moduli. 
 \end{proof}

\begin{rem}
As explained in the introduction of \cite{rana-urzua17}, the log Bomolov-Miyaoka-Yau inequality gives $d\le 34$ for a stable I-surface with a T-singularity of type $\frac {1}{4d}(1,2d-1)$, a weaker bound than Proposition \ref{prop: D} (i).
\end{rem}
\begin{rem}
Proposition \ref{prop: D} shows that the expectation that T-singular surfaces of type $\frac{1}{4d}(1,2d-1)$  give a codimension $d$ subset in the moduli space
(see \S \ref{sec: T-sing}) is true for $d=1,2, 3$ but not for all possible $d$. In fact, for $d=25$ one has a family depending on 4 moduli while   the expected number is $28-25=3$.  \end{rem}

\begin{lem}\label{lem: g=2} Let $C$ be a smooth genus 2 curve and let $P\in C$ be a Weierstrass point. Then  there is an embedding $j\colon C\to\F_2$ such that
  $j(C)\sim 2\sigma_0+\Gamma$  and $j(C)$ intersects $\sigma\inff$ at $j(P)$. 
\end{lem}
\begin{proof}
The canonical double cover  $C\to\pp^1$ gives a natural embedding of $C$ in the total space $V$  of the line bundle $\OO_{\pp^1}(-3)$. In turn, there is an open immersion $V\hookrightarrow\F_3$ 
 that identifies $V$ with  the complement of the infinity section $\sigma\inff$. Composing these two inclusions one gets an  inclusion of  $C$ in $\F_3$ as a bisection disjoint from $\sigma\inff$. 
 Blowing up $\F_3$ at the Weierstrass point $P$ and contracting the strict transform of the ruling of $\F_3$ containing $P$,  one obtains the desired inclusion $j\colon C\to \F_2$. 
\end{proof}
\subsection{The case of index $n=3$}
Here we construct T-singular I-surfaces of type $\frac{1}{18}(1,5)$. 
We start by proving an auxiliary   result on elliptic surfaces. 
\begin{lem}\label{lem: elliptic-3} Let $Y$ be a minimal elliptic surface with $p_g(Y)=2$ and $q(Y)=0$. If $Y$ contains a $(-3)$-curve $B$ then: 
\begin{enumerate}
\item   $Y$ is the minimal resolution of a double cover $\pi\colon \bar Y\to \F_6$ branched on a divisor $D\in |\sigma\inff+3\sigma_0|$ with at most  negligible singularities and $\sigma_{\infty} $ pulls back to $2B$ on $Y$;
\item  $Y$ has no multiple fibers;  the reducible fibers of the elliptic fibration $Y\to \pp^1$ are the preimages of the rulings of $\F_6$ containing a singular point of $D$.
\end{enumerate}
Conversely, the minimal resolution $Y$ of a double cover $\bar Y\to \F_6$ as in (i) is a minimal elliptic surface with $p_g(Y)=2$, $q(Y)=0$ and the pull-back of $\sigma_{\infty}$ to $Y$ is equal to $2B$ for   a  $(-3)$-curve $B$. 
\end{lem}
\begin{proof}
(i)+(ii)  Denote by $F$ a general fiber of the elliptic fibration $Y\to\pp^1$. By the canonical bundle formula for elliptic surfaces we have $|K_Y|=|aF| +\sum (m_i-1)F_i$, where  $m_1F_1, 
\dots m_kF_k$ are  the multiple fibers. Since $p_g(Y)=2$ we have $a=1$; since $K_YB=1$, we conclude that $B$ is a section of the elliptic fibration and that there are no multiple fibers. 
Set $L:=2B+7F$; one has $LB=1$ and  $L^2=16$,   $K_YL=2$.  We write $L=K_Y+(6F+2B)$; since $6F+2B$ is nef and big, Kawamata-Viehweg vanishing applies and $h^0(L)=\chi(L) =10$.  A similar argument shows that $H^1(7F+B)=0$ and therefore the restriction map $H^0(L)\to H^0(\OO_B(1))$ is surjective and  the linear system $|L|$ is base point free. Let $\fie\colon Y\to \pp^9$ be the morphism defined by $L$ and let $\Sigma$ be the image of $\fie$.  The morphism  $\fie$ maps a general $F$ 2-to-1 onto a line and it maps $B$   to a line $r$ that meets the  images of the elliptic fibers of $Y$ at distinct points. So the degree $m$   of $\fie$ is equal to $2$ and  $\deg\Sigma=L^2/2=8$.  By the classification of surfaces of minimal degree   in $\pp^N$ the surface is either a cone over the rational normal curve of degree 8 or  a smooth linear scroll over $\pp^1$. Since   $\Sigma$  contains the  line $r$ that meets each  ruling  at a  distinct point, we conclude that  $\Sigma\cong \F_6$ and   the line $r$  corresponds to  the infinity section $\sigma\inff$ of $\F_6$. In addition, it is not hard to see that $B$ is contained in the ramification locus of $\fie$. 

Let $Y\to \bar Y\overset{\pi}{\to} \F_6$ be the Stein factorization of $f$: the map $Y\to \bar Y$ contracts  precisely the $(-2)$-curves of $Y$ that do not meet $B$, so $\bar Y$ has canonical singularities. The map $\pi$ is a flat double cover; we write $D:=\sigma\inff  +D_1$. Since the preimage  of a general ruling $\Gamma$  of $\F_6$ is an elliptic curve and $D$ is divisible by 2 in $\pic(\F_6)$, we may write  $D_1=3\sigma\inff+2a\Gamma$. 
The usual formulae for double covers give  \[K_{\bar Y}=\pi^*(K_{\F_6}+2\sigma\inff+a\Gamma)=\pi^*((a-8)\Gamma).\] 
Since $p_g(\bar Y)=p_g(Y)=2$, we obtain $a=9$ and $D_1\in |3\sigma_0|$. The singularities of $Y$ occur above the singularities of $D$, which are  therefore  negligible because $Y$ has canonical singularities.
\medskip

Conversely, given a cover as in (i), $\bar Y$ has canonical singularities and  is smooth above $\sigma\inff$. If we write $\pi^*\sigma\inff =2B$ we get $-12=2\sigma\inff^2=4B^2$, namely $B^2=-3$ and $B$ is a $(-3)$-curve.  As noted above, the ruling of $\F_6$ pulls back to a pencil $|F|$ of elliptic curves  and  the curve $B$ is a section of $|F|$,  therefore  $|F|$ has  no multiple fibers. Denote by $\bar Y\to Y$ the minimal desingularization. Then the same computations as before give $p_g(Y)=p_g(\bar Y)=2$. 
 Since $\sigma\inff$ and $D_1$ are disjoint, the restriction of $D$ to a ruling of $\Gamma$ cannot be divisible by 2. So the strict transform in $Y$ of a ruling of $\F_6$ is always irreducible and the components of a reducible fiber  that do not meet $B$ are precisely the exceptional curves of $\bar Y\to Y$.
\end{proof}

\begin{example}\label{ex: (b)}
Let $Y$ be an elliptic surface with $p_g(Y)=2$, $q(Y)=0$ such that:
\begin{itemize}
\item $Y$ has a $(-3)$-section  $B$
\item $Y$ has  an  $I_2$ fiber $F_2$, and all the remaining fibers are irreducible.
\end{itemize}
By Lemma \ref{lem: elliptic-3}, the surface $Y$ is the minimal resolution of a  surface $\bar Y$ which is a double cover $\pi\colon \bar Y\to \F_6$ branched on a divisor $D\in |\sigma\inff+3\sigma_0|$ with an ordinary double point $P$ and no other singularity. The $I_2$ fiber arises as the pull-back of the ruling  of $\F_6$ through $P$ and $B$ is the preimage  of $\sigma_{\infty}$.

Consider an  irreducible singular fiber $F_1$, of type either $I_1$ or $II$,  and let $Q\in F_1$ be the singular point. 
Let  $\tilde Y\to Y$ be the blow-up at the singular point $Q$,  denote by $A$ the strict transform of $F_1$ and by $C$ 
 (the strict transform of) the component of $F_2$ that meets $B$ (see Figure \ref{fig:  index 3 surface}).
 Then $A,B,C$ is a string of type $[4,3,2]$ that can be blown down to obtain a surface $X$ with a singularity of
 type $\frac{1}{18}(1,5)$,  $K^2_X=1$, $p_g(X)=2$ and  $q(X)=0$ (see Remarks \ref{rem: T-surf} and \ref{rem: pg constant}). 
 The pull-back of $K_X$ to $\tilde  Y$ is equal to $K_{\tilde Y}+\frac 23A+\frac23B+\frac13C$. It is not hard to check that it is a nef divisor and 
  that $A,B$ and $C$ are the only curves that have zero intersection with it. So $K_X$ is ample and $X$ is a  stable T-singular surface of type $\frac{1}{18}(1,5)$.


\end{example}
\begin{rem}\label{rem: moduli-n=3}
The construction of Example \ref{ex: (b)} depends on 27 moduli. Indeed,  the pair $(Y, B)$ determines $X$ up to a finite number of possibilities, so it is enough to count parameters for the pairs $(Y,B)$ with an $I_2$ fiber. These are determined  by the branch locus $D$ of $\pi\colon \bar Y\to \F_6$, that has a double point.
The linear system $|3\sigma_0|$ has dimension 39, and the  curves with a double point give a codimension 1 subvariety.
 Since the automorphism group of $\F_6$ has dimension 11, we are left with $38-11=27$ parameters.
\end{rem}

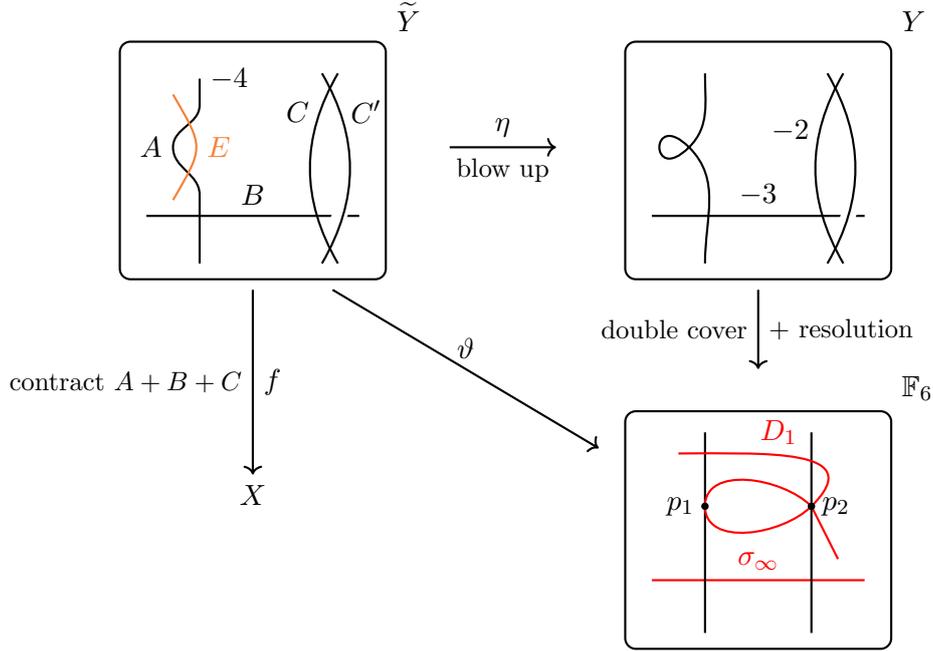
\begin{figure}
\begin{tikzpicture}
[curve/.style ={thick, every loop/.style={looseness=10, min distance=30}},
exceptional/.style = {Orange},
scale = 0.7
]

\begin{scope}[xshift=-9.5cm]
\draw[thick, rounded corners] (-2.5, -1.5) rectangle (2.5, 3) node [above right] {$\tilde Y$};
\draw [curve,->] (0, -1.7) to node [right] { $f$} node [left] {\small contract $A+B+C$} ++(0,-3.5) node [below] {$X$} ;
\begin{scope}[curve]
\draw  (-2,-0.3)  to  node[above]{$B$} (2,-0.3);
\draw[line width = .2cm,   white] (1.3, -1) to[bend right]  ++(0,3.6);
\draw (1.3, -1.2) to[bend right] node[pos = 0.8, right]{$C'$} ++(0,3.6);
\draw (1.6, -1.2) to[bend left] node[pos = 0.8,left] {$C$} ++(0,3.6);
\draw  (-1,-1.2) to ++ (0, 1.3)  to[out=90, in = -90] ++ (-0.5,0.9) node[left] {$A$} to[out=90, in = -90] ++ (0.5,0.8) to ++(0,0.5) node [right] {$-4$};
\draw[exceptional] (-1.5, 0) to [bend right, looseness = 1.5] node[right] {$E$} ++(0,2);

\end{scope}
\draw [curve,->] (3.7, 1) to node[above] {$\eta$} node[below] {\small blow up} ++(2,0);
\end{scope}

\begin{scope}
\draw[thick, rounded corners] (-2.5, -1.5) rectangle (2.5, 3) node [above right] {$Y$};
\begin{scope}[curve]
\draw [curve] (-2,-0.3)  to node [above] {$-3$}  (2,-0.3);
\draw[line width = .2cm,   white] (1.3, -1.2) to[bend right] ++(0,3.6);
\draw (1.3, -1.2) to[bend right] ++(0,3.6);
\draw (1.6, -1.2) to[bend left] node[pos = 0.7,left] {$-2$}++(0,3.6);
\draw [curve] (-1,-1.2) to[out=90, in=-45] (-1.3, 1) to[out=135, in =225,loop] () to[out=45, in = 270]  (-1,2.4);
\end{scope}
\draw [curve,->] (0, -1.7) to  node[left]{\text{\small double cover}} node[right]{\text{\small + resolution}} ++(0,-1.5);
\draw[curve, -> ](-8, -1.7) to node[above] {$\theta$} ++(5, -3);
\end{scope}

\begin{scope}[yshift = -7cm]
\draw[thick, rounded corners] (-2.5, -1.5) rectangle (2.5, 3) node [above right] {$\IF_6$};
\begin{scope}[curve]
\draw [curve, red] (-2,-0.2)  to node [above] {$\sigma_\infty$}  (2,-0.2);
\draw [curve] (-1,-1.2) to  ++ (0,3.8);
\draw [curve] (1,-1.2) to  ++ (0,3.8);
\coordinate (p1) at (-1, 1.2);
\coordinate (p2) at (1, 1.2);
\draw[curve, red] (p2) ++ (.5, -1) -- (p2) to[out = 135, in = 90] (p1) to[out = -90, in = -135] (p2) to[out = 45, in = 0, looseness = 1.5]  node [pos = .6, above] {$D_1$}(-1.5, 2.2);
\fill (p1) circle [radius = 2pt] node[left] {$p_1$};
\fill (p2) circle [radius = 2pt] node[right] {$p_2$};
\end{scope}
\end{scope}

\end{tikzpicture}

\caption{Construction of an I-surface of type $\frac{1}{18}(1,5)$, using a nodal fibre.}\label{fig:  index 3 surface}
\end{figure}

In order to show that the surfaces constructed in Example \ref{ex: (b)} are smoothable, we give an alternative description by computing their canonical ring: 
\begin{prop}\label{prop: can ring index 3}
 Let $X$ be as in Example \ref{ex: (b)}. Then there exist sections $x_1, x_2\in H^0(X, K_X)$, $y\in H^0(X, 2K_X)$, $u\in H^0(X,3 K_X)$ and $z\in H^0(X, 5K_X)$ such that the canonical ring of $X$ is
 \[ R(X, K_X) = \IC[x_1, x_2, y, u, z]/ (x_1^3- x_2y, z^2 -f_{10}(x_1, x_2, y, u))\]
 for a weighted homogeneous polynomial $f_{10}$ of degree $10$.
 
 In particular, $X$ is embedded into $\IP(1,1,2,3,5)$ as a complete intersection of degree $(3,10)$.
\end{prop}
\begin{proof}
Recall that $p_g(X)= 2$; in addition, one can check that the correction term  $\frac 12\{m\Delta\}\left(\{m\Delta\}-\{\Delta\}\right)$ vanishes for all values of $m\ge 2$, so that \eqref{eq: plurigenus-lt}  gives:
\begin{equation}\label{eq: hR}
h^0(X, mK_X) = 3+ \frac{m(m-1)}2 \qquad (m\geq 2).
\end{equation}
We want to compute the canonical ring using the identification $H^0(X, mK_X) =H^0(\tilde Y, \lfloor \pi^* mK_X\rfloor)$. We want to relate these to linear systems on $\IF_6$ as in the diagram, where we already added some information on the $3$-canonical map explained below:
\begin{equation}\label{diag: 3 canonical index 3}
 \begin{tikzcd}
  \tilde Y \rar{\eta} \dar{f}\arrow[bend left]{rrr}{\theta} & Y \rar[swap]{\text{resolution}}&\bar Y\rar{\pi}  & \IF_6 \rar{|\sigma_0|} & \kq_6 \rar[hookrightarrow]\dar[dashed] & \IP^7\dar[dashed]{\text{projection from line}}\\
  X\arrow[dashed]{rrrr}{|3K_X|} &&&& \kq_4 \rar[hookrightarrow] & \IP^5
 \end{tikzcd}.
\end{equation}
For later reference we compute the relevant divisors on $\tilde Y$. As above we denote by $\Gamma$  a ruling on $\IF_6$ and by $C'$ the $-2$-curve such that $C+C'$ is a fiber of the elliptic fibration of $\tilde Y$. The configuration is depicted in Figure  \ref{fig: index 3 surface}. 
 \begin{align*}
  f^*K_X  & = \theta^*\Gamma + E + \frac13\left( 2A + 2B + C\right),\\
  f^*2 K_X  & = \theta^*3\Gamma + B + \frac13\left( A + B + 2C\right),\\
  f^*3 K_X  & = \theta^*( 6\Gamma + \sigma_\infty) - E - C'\\
  & = \theta^*\sigma_0-E-C',\\
  f^*4 K_X  & = \theta^*( 7\Gamma + \sigma_\infty) - C'+ \frac13\left( 2A + 2B + C\right),\\
  f^*5 K_X  & = \theta^*( 9\Gamma + \sigma_\infty) - E - C' +B +\frac13\left( A + B + 2C\right).
 \end{align*}
 Note that by construction the branch divisor $D$ of $\theta$ is equal to $\sigma_{\infty}+D_1$, with  $D_1\in |3\sigma_0| \subset |4\sigma_\infty + 18 \Gamma|$,   so that 
 \[ H^0(\tilde Y , \theta^*(a\Gamma + b\sigma_\infty)) \isom H^0(\IF_6, a\Gamma + b\sigma_\infty)\oplus H^0(\IF_6, (a-9)\Gamma + (b-2) \sigma_\infty),\]
 which is the decomposition into the invariant and anti-invariant part. We are ready to compute the relevant pluricanonical systems on $Y$, but for the ring structure we also need the multiplication maps. Considering these on $\tilde Y$ we need to account for correction 
terms, for example,
\[
 \begin{tikzcd}
  H^0(\lfloor f^* K_X\rfloor) \times H^0(\lfloor f^* K_X\rfloor)  \rar&  H^0(2\lfloor f^* K_X\rfloor)\rar{+A+B} &   H^0(\lfloor 2f^* K_X\rfloor)
 \end{tikzcd}.
\]
 
 We now compute the pullback of the canonical ring to $\tilde Y$. 
 Let us denote the section of a line bundle associated to a curve by the corresponding lower case letter.
Then $H^0(X, K_X) = H^0(\tilde Y, \theta^* \Gamma+E)  = e\cdot  H^0(\tilde Y, \theta^*\Gamma)$, where the second equality is most easily confirmed by dimension reasons. Thus the canonical pencil is spanned by 
\[
x_1 =cc'e\text{ and } x_2 = (ae^2)e                                                                   
\]
Taking the correction into account, the image of the multiplication map is spanned by 
$\langle x_1^2 = (cc')^2  ae^2 b, x_1x_2 = cc' (ae^2)^2b, x_2^2 = (ae^2)^3b\rangle$, which together with
\[y = (cc')^3b\]
forms a basis of $H^0(\tilde Y , \lfloor2f^*K_X\rfloor)= H^0(\tilde Y,  \theta^* 3\Gamma + B) = b\cdot \theta^* H^0(\IF_6, 3\Gamma)$.
Looking at the next multiplication map 
\[
 \begin{tikzcd}
  H^0(\lfloor f^* K_X\rfloor) \times H^0(\lfloor2 f^* K_X\rfloor)  \rar\arrow{dr}&  H^0(\lfloor f^* K_X\rfloor+\lfloor2 f^* K_X\rfloor)\dar{+A+B+C} \\ &    H^0(3f^* K_X)
 \end{tikzcd}.
\]
we find the claimed relation $x_1^3 = x_2y$. Because of this relation, the image of the multiplication map is of dimension $5$ and we need a further generator $u\in H^0(X, 3K_X)$, not contained in the image.
 
 \begin{rem}
 It is now instructive to look at the $3$-canonical map, as alluded to in Diagram \ref{diag: 3 canonical index 3}. 
  On $\tilde Y$ we have 
 \[H^0(\tilde Y , 3f^*K_X) = H^0(\tilde Y, \theta^* \sigma_0 -E-C') \subset \theta^*H^0(\IF_6, \sigma_0).\]
 Note that $\theta$ maps $E$ and $C'$ to points $p_1, p_2\in \IF_6$ such that $D_1$ is tangent to a ruling in $p_1$ and $D_1$ has a node at $p_2$, necessarily on a different ruling, see Figure \ref{fig:  index 3 surface}. 
 Thus the 3-canonical  system is precisely $\theta^*H^0(\IF_6, \ki_{\{p_1, p_2\}} (\sigma_0))$. Since $|\sigma_0|$ maps $\IF_6$ to the cone $\kq_6$ over the rational normal curve of degree $6$ in $\IP^7$, the image of the $3$-canonical map of $X$ is the projection of $\kq_6 $ from a line through two general points of $\kq_6$, which is the cone $\kq_4\subset \IP^5$ over the rational normal curve of degree $4$. 
 
 In this description, $u$ is the preimage of any hyperplane section of $\kq_6$ containing $p_1,p_2$ and not containing the vertex.
 \end{rem}
 
 We have thus found the subring $S$ generated by elements of degree at most $3$ in the canonical ring:
 \[ R:=  R(X, K_X) \supset \IC\left[ H^0(mK_X)\colon m\leq 3\right] = \IC[x_1, x_2, y, u]/(x_1^3-x_2y) = :S\]
To ease computations later, recall that the Hilbert series of a weighted polynomial ring $\IC[w_1, \dots, w_r]$ with weights $d_1, \dots, d_r$ is given by $\prod_{i=1}^r \left( 1- t^{d_i}\right)^{-1}$ and by the additivity of Hilbert series we get the Hilbert functions for complete intersections. In particular,
\begin{align*}&h_S(t) = \frac{1}{(1-t)^2(1-t^2)}= \frac{(1-t^3)}{(1-t)^2(1-t^2)(1-t^3)}, \\
 &h_R(t) = \frac{(1-t^{10})}{(1-t)^2(1-t^2)(1-t^5)} = \frac{(1-t^3)(1-t^{10})}{(1-t)^2(1-t^2)(1-t^3)(1-t^5)},
\end{align*}
 because by \eqref{eq: hR} the Hilbert function of $R$ coincides with the Hilbert function of the canonical ring of  a smooth I-surface, which in turn is a complete intersection of degree 10 in $\pp(1,1,2,5)$ (cf. \ref{ssection: I-surf}).
 
 Note that the surface $X$ carries an involution $\iota$, which acts on the canonical ring and that birationally $\iota$ is the covering involution from the double cover of $\pi\colon \bar Y \to \IF_6$.
Thus the whole subring $S$ is invariant under the involution.

A  dimension computation gives $S_4 = R_4$, so we now want to analyse the $5$-canonical system, which is of dimension $13$, while $S_5$ is of dimension $12$. Since all sections in $S_2$ are divisible by $b$, we see that 
\[ S_5 =  S_2\cdot S_3 \subset H^0(\tilde Y , \lfloor 5 f^*K_X\rfloor- B) \isom H^0(\IF_6,  \ki_{\{p_1, p_2\}}(9\Gamma+\sigma_\infty)),\]
and in fact equality holds as both sides are of dimension $12$. 
 We claim that this is indeed the $\iota$-invariant part of the linear system $|\lfloor 5K_X\rfloor|$, in which it is clearly contained.

For this consider the singular double cover  $\pi \colon  \bar Y \to \IF_6$ branched over $D_1 + \sigma_\infty$ occuring in the Stein factorisation of $\theta$ and write 
$\pi^*(D_1+\sigma_\infty) = 2R_1 + 2 B$. Then 
\[\pi^*(9\Gamma+ 2 \sigma_\infty) - B =\frac 12 \pi^*(D_1 + \sigma_\infty) - B\sim (R_1+B)-B.\]
If $\rho$ is the section defining the divisor $R_1+ B$ we have
\[H^0(\bar Y, \pi^*(9\Gamma + 2 \sigma_\infty))  =\pi^*H^0(\IF_6,   9\Gamma + 2 \sigma_\infty)\oplus \langle \rho \rangle\]
as a decomposition into $\iota$-eigenspaces. 
Since the pull back of $\rho$ to $\tilde Y$ vanishes along $B$, $E$ and $C'$ it defines an anti-invariant element of $\lfloor 5K_X\rfloor$, which we call $z$. 
By the restriction sequence, $z$ restricts to a non-zero constant section on $B$. 

We claim that $R_{10} = S_{10} \oplus z\cdot S_5$. Computing the dimensions, e.g., using the Hilbert functions, we see that the dimensions match on both sides. 
The intersection of the two subspaces is zero, since by our choice of $z$, one is invariant and one is anti-invariant under the action of the involution. 

So since $z^2$ is now an invariant section there is a relation of the form $z^2- f_{10}(x_1, x_2, y, u)$ in $R_{10}$. 
 We conclude that we have an injection
 \[ S[z]/(z^2-f_{10}) \into R\]
 which has to be an isomorphism because both rings have the same Hilbert function.
 In total, the canonical ring of $X$ has the claimed format.
\end{proof}

\begin{cor}\label{cor: index 3 smoothable}
 A surface $X$ as in Example \ref{ex: (b)} is smoothable. 
\end{cor}
\begin{proof}
Write the canonical ring of $X$ as in Proposition \ref{prop: can ring index 3}. Now consider the family
 $\kx \subset \IP(1,1,2,3,5)\times \IA^1_t\to \IA^1 = \kb$ cut out by the equations.
\[tu - x_1^3+x_2y, z^2 -f_{10}(x_1, x_2, y, u)+tg_{10}\]
where $g$ is a general  homogenous polynomial of degree $10$. 

Note that the general fibre is a smooth I-surface, as if $t\neq0$ we can eliminate the variable $u$ to get a hypersurface of degree $10$ in $\IP(1,1,2,5)$. When we set $t=0$ we find the equations of our surface  $X = \kx_0$. 

Clearly the family is flat over the curve $\kb$, because every fibre is a surface, i.e., every component of $\kx$ dominates $\kb$. It is also a $\IQ$-Gorenstein smoothing because Koll\'ar's condition
\[ \ko_{\kx}(m)|_{\kx_0} =  \omega_{\kx/\kb}^{[m]}|_{\kx_0} \overset{\isom}{\to}  \omega^{[m]}_{\kx_0} = \ko_{\kx_0}(m)\]
is met; 
this is equivalent to $\ko_\kx(m)$ being flat over $\kb$ by \cite{Kollar13}. \end{proof}

\subsection{The case of index $n=5$}
An example of a T-singular  I-surface of type  $\frac{1}{25}(1,14)$  can be found in \cite{rana-urzua17} right after the proof of Thm.~3.2. For the reader's convenience we recall here its description, visualized in Figures \ref{fig: RU surface} and \ref{fig: RU surface cuspidal}.
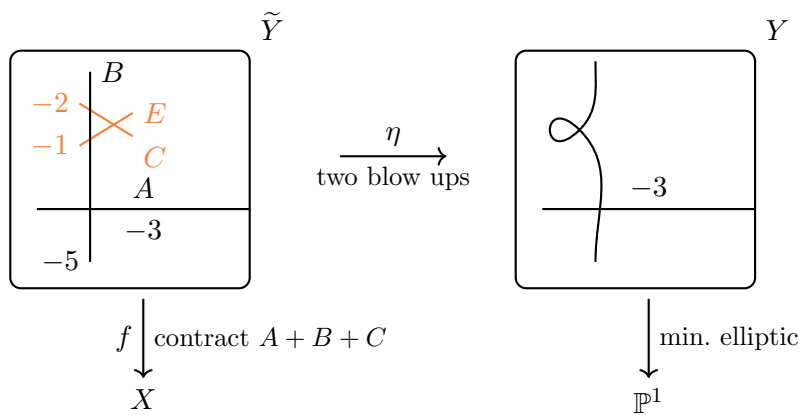
\begin{figure}
\begin{tikzpicture}
[curve/.style ={thick, every loop/.style={looseness=10, min distance=30}},
exceptional/.style = {Orange},
scale = 0.7
]

\begin{scope}[xshift=-9.5cm]
\draw[thick, rounded corners] (-2.5, -1.5) rectangle (2, 3) node [above right] {$\tilde Y$};
\draw [curve, exceptional] (-1.2,1.2) node[below, left]{$-1$} to  ++ ( 1, .62) node[ right]{$E$};
\draw [curve, exceptional] (-1.2,2) node[above,left ]{$-2$} to  ++(1,-.62) node[ below right]{$C$};
\draw [curve] (-1,-1) node[left]{$-5$} to  (-1,2.6) node[right] {$B$};
\draw [curve] (-2,0)  to node [below] {$-3$}  node[ above] {$A$} (2,0);
\draw [curve,->] (3.7, 1) to node[above]  {$\eta$} node[below] {\small two blow ups} ++(2,0);
\draw [curve,->] (0, -1.7) to node [left] { $f$} node [right] {\small contract $A+B+C$} ++(0,-1.5) node [below] {$X$} ;
\end{scope}

\begin{scope}
\draw[thick, rounded corners] (-2.5, -1.5) rectangle (2, 3) node [above right] {$Y$};
\draw [curve] (-1,-1) to[out=90, in=-45] (-1.3, 1.5) to[out=135, in =225,loop] () to[out=45, in = 270]  (-1,2.8);
\draw [curve] (-2,0)  to node [above] {$-3$}  (2,0);
\draw [curve,->] (0, -1.7) to  node[right]{\small min.\ elliptic} ++(0,-1.5) node [below] {$\IP^1$};
\end{scope}
\end{tikzpicture}

\caption{Construction of an RU surface ($\frac{1}{24}(1,4)$ singularity), nodal case}\label{fig: RU surface}
\end{figure}

\begin{figure}
\begin{tikzpicture}
[curve/.style ={thick, every loop/.style={looseness=10, min distance=30}},
exceptional/.style = {Orange},
scale = 0.7
]

\begin{scope}[xshift=-9.5cm]
\draw[thick, rounded corners] (-2.5, -1.5) rectangle (2, 3) node [above right] {$\tilde Y$};
\draw [curve, exceptional] (-1.4,1.8) node[below, above left]{$-1$} to  ++( 1.2, -.6) node[ right]{$E$};
\draw [curve, exceptional] (-1.4,1.4) node[above,below left ]{$-2$} to  ++(1.2, .6) node[right]{$C$};
\draw [curve] (-1,-1) node[left]{$-5$} to  (-1,2.6) node[right] {$B$};
\draw [curve] (-2,0)  to node [below] {$-3$}  node[ above] {$A$} (2,0);
\draw [curve,->] (3.7, 1) to node[above]  {$\eta$} node[below] {\small two blow ups} ++(2,0);
\draw [curve,->] (0, -1.7) to node [left] { $f$} node [right] {\small contract $A+B+C$} ++(0,-1.5) node [below] {$X$} ;
\end{scope}

\begin{scope}
\draw[thick, rounded corners] (-2.5, -1.5) rectangle (2, 3) node [above right] {$Y$};
\draw [curve] (-1,-1) to ++(0, 1) to[out = 90, in = 0] ++( -.65, 1.5) to[out = 0, in = -90] (-1, 2.8);
\draw [curve] (-2,0)  to node [above] {$-3$}  (2,0);
\draw [curve,->] (0, -1.7) to  node[right]{\small min.\ elliptic} ++(0,-1.5) node [below] {$\IP^1$};
\end{scope}
\end{tikzpicture}

\caption{Construction of an RU surface ($\frac{1}{24}(1,4)$ singularity), cuspidal case}\label{fig: RU surface cuspidal}
\end{figure}
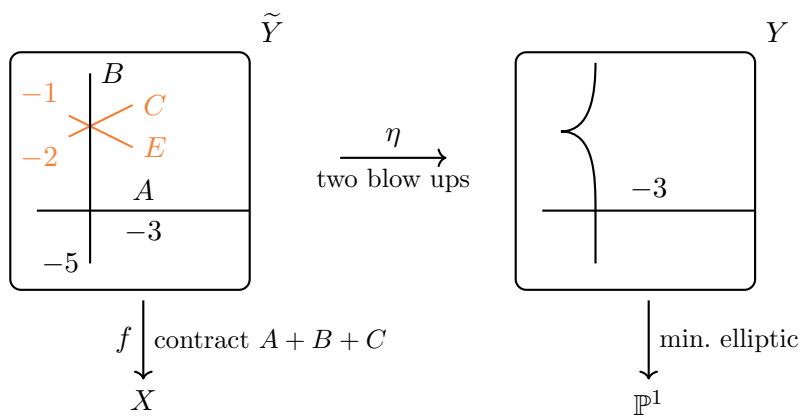

\begin{example}\label{ex: (c)}
Let $Y$ be an elliptic surface with $p_g(Y)=2$, $q(Y)=0$ such that:
\begin{itemize}
\item $Y$ has a $(-3)$-section  $A$
\item  all the elliptic   fibers are irreducible. 
\end{itemize}
By Lemma \ref{lem: elliptic-3}, the surface $Y$ is a double cover $\pi\colon  Y\to \F_6$ branched on a smooth divisor $D\in |\sigma\inff+3\sigma_0|$. 

Let $F_1$ be a singular  fiber and let $Q$ be its singular point: blow up $F_1$ at $Q$ and then at a point $Q_1$ infinitely near to $Q$ and lying on the strict transform of 
$F_1$ to get a surface $\tilde Y$.
 The strict transform of $F_1$ is a $(-5)$ curve $B$, the strict transform of $A$ (that we still denote by $A$) is a $(-3)$-curve,
   the strict transform of the curve of the first blow up is a $(-2)$-curve, which we call $C$, so that  $A,B,C$ is a string of type $[3,5, 2]$ (note that this is true both for  $F_1$ nodal and for $F_1$ cuspidal).  Then the string $A,B,C$ can be blown down to obtain a T-singular surface $X$ of type $\frac{1}{25}(1,14)$ with $K^2_X=1$, $p_g(X)=2$, $q(X)=0$. The pull-back of $K_X$ to $\tilde  Y$ is equal to $K_{\tilde Y}+\frac 35A+\frac45B+\frac25C$. It is not hard to check that it is a nef divisor and  that $A,B$ and $C$ are the only curves that have zero intersection with it. So $K_X$ is ample and $X$ is a    T-singular surface of type $\frac{1}{25}(1,14)$.
\end{example}

\begin{rem}\label{rem: moduli-n=5}
Counting parameters  as in Remark \ref{rem: moduli-n=3}, we obtain 28 moduli for the construction in Example \ref{ex: (b)}. Since  the closure of the  locus of smooth I-surfaces is irreducible of dimension $28$, we conclude that the general surface obtained via this construction is not smoothable. This confirms the infinitesimal computations of \cite{rana-urzua17}, where it is shown that  the obstruction space for $\Q$-Gorenstein deformations is non-zero for these surfaces. 
\end{rem}
The above remark can be made more precise as follows:
\begin{prop}\label{prop: 5 I_1 not smoothable}
Let $X$ be a T-singular surface obtained as in Example \ref{ex: (c)} taking as $F_1$ an irreducible fiber of type $I_1$. 
Then $X$ is not smoothable  and such surfaces give a dense open subset of an irreducible component of $\overline\gothM_{1,3}$ of dimension 28. 
\end{prop}
\begin{proof} 
Assume by contradiction that $X$ is smoothable. Then  by Corollary  \ref{cor: smoothable-invo}, $X$ has an involution that lifts to an involution $\tau$ of $\tilde Y$ preserving the exceptional curves $A$, $B$ and $C$. In addition, $\tau$ maps to itself the exceptional curve $E$ of the second blow up of $Y$, since $E$ is the only irreducible $(-1)$-curve of $\tilde Y$. So $\tau$ maps $B$ to itself and  fixes the three distinct points  $B\cap A$, $B\cap C$ and $B
\cap E$. Since $B$ is a smooth rational curve, $\tau$ restricts to the identity on $B$ and therefore the induced involution $\bar \tau$ of $Y$ fixes the singular  fiber $F_1$ pointwise.  This is impossible since the divisorial part of the fixed locus of an involution on a smooth surface is a smooth curve.

To show that these surfaces give an open subset of an irreducible component of the moduli space it is enough to show that every small deformation of such a surface is equisingular, that is, also contains a T-singularity of the same type. 
Assume for contradiction that we have a non locally trivial deformation of such an $X$. Then, since every non-trivial  deformation of the singularity $\frac 1{25}(1,14)$
 has canonical singularities  \cite[Prop. 2.3]{Hacking-Prokhorov10}, $X$ would deform to a canonical surface and hence be smoothable --- a contradiction.
\end{proof}
\begin{rem}\label{rem: conjecture cuspidal}
 Clearly, the construction from Example \ref{ex: (c)} using a nodal fibre degenerates to the one constructed with a cuspidal fibre. Preliminary computations suggest that the latter surfaces might be smoothable.
 \end{rem}

\section{The classification: proof of Theorem \ref{thm: oneT} and Corollary \ref{cor: moduli}}
Throughout this section $X$ is a T-singular I-surface with a singularity of type $\frac{1}{dn^2}(1,dna-1)$; we use freely the notation of \S \ref{sec: stable-surf}.

\begin{lem}\label{lem: Y-ell}
The surface $Y$ is properly elliptic with $p_g(Y)=2$, $q(Y)=0$ and there are the following cases to consider:
\begin{center}
 \begin{tabular}{ccccc}
  \toprule
  $r-d$ & $n$ & $K_{\tilde Y}^2$ & T-singularity& T-string\\
  \midrule
  $0$ & $2$ & $0$ & $ \frac{1}{4d}(1,2d-1)$& $[4]$ or $[3,3]$ or $[3, 2\dots, 3]$\\
  $1$ & $3$ & $-1$ & $ \frac{1}{18}(1,5)$&  $[4,3,2]$\\
  $2$ & $5$ & $-2$ & $\frac{1}{25}(1, 14)$ & $[2,5,3]$\\
  \bottomrule
 \end{tabular}

\end{center}

\end{lem} 
\begin{proof}
Since T-singularities are rational, $p_g(Y)=p_g(\tilde Y)=p_g(X)=2$ and $q(Y)=q(\tilde Y)=q(X)=0$ (Remark \ref{rem: pg constant}) and $K^2_{\tilde Y} =d-r$ by \eqref{eq: lee formula}. In addition we have $K^2_Y<K^2_X =1$ by 
\eqref{eq: K2}, hence $K^2_Y=0$ and $Y$ is properly elliptic.

By~\cite[Theorem 1.1]{rana-urzua17} we have $r-d\le 2$ and if $r-d=2$, then by~\cite[Theorem 3.2]{rana-urzua17}, the singularity must be of type $\frac{1}{25}(1,14)$, giving the third row of the table.

If $r-d=1$, then $n=3$, $K_{\tilde Y}^2=-1$ and the T-string is either
$[5,2]$ (the $\frac19(1,2)$ singularity)
or $[4,2,\ldots, 2, 3,2]$, where there are $d-2$ curves of self-intersection $(-2)$ between the $(-3)$-curve and $(-4)$-curve (the $\frac {1}{9d} (1,3d-1)$ singularities). 
We show that the former cannot occur, and the latter is possible only if $d=2$, i.e. the chain is $[4,3,2]$ and the singularity is $\frac {1}{18} (1,5)$.  

Notice that because $K_{\tilde Y}^2=-1$ and $Y$ is a minimal elliptic surface, the surface $\tilde{Y}$ contains exactly one $(-1)$-curve $E$ that we contract to obtain $Y$. By ampleness of $X$, the curve $E$ must intersect the T-string at least twice, and by nefness of $K_Y$ the curve $E$ cannot intersect a $(-2)$-curve. 

Let us suppose that the T-string is $[5,2]$. Denote by $A$ the $(-5)$-curve, and by abuse of notation its image in $Y$. Then as we have just argued, we have $EA\ge 2$. On the other hand, $EA\le 3$, because otherwise we would have $AK_Y<0$.  If $AE=3$, then the curve $A$ in $Y$ has a triple point. On the other hand, we have by adjunction that $AK_Y=0$, so $A$ must also be contained in a fiber of the elliptic fibration, a contradiction. If instead $AE=2$, then $A$ has a double point and by adjunction $AK_Y=1$. This means $A$ is a section of the fibration with a double point, which is impossible.

Now suppose that the T-string is $[4,2,\ldots, 2, 3,2]$. Then $E$ cannot intersect the T-string more than three times, since otherwise one of the curves in the T-string would become $K_Y$-negative. Denote by $A$ and $B$ the curves in the T-string of self-intersections $(-4)$ and $(-3)$, respectively, and by abuse of notation, their images in $Y$. Notice that $EA\le 2$ and $EB\le 1$ because otherwise $A$ or $B$ becomes $K_Y$-negative.

Then we have three possibilities:
\begin{enumerate}
\item[1)] $EA=EB=1$, or
\item[2)] $EA=2$, $EB=1$, or
\item[3)] $EA=2$, $EB=0$.
\end{enumerate}

In case 1), upon contracting $E$, we see that $A$ becomes a $(-3)$-curve on $Y$, so by adjunction we have $K_Y A=1$. On the other hand, the $(-2)$-curves in the T-chain and the curve $B$
 become parts of a fiber on $Y$, so $A$ also intersects an $I_k$ fiber twice for some $k\ge 1$. This forces $A$ to be a multisection. But the canonical bundle formula together with $p_g=2$ implies that this is impossible.

In cases 2) and 3), contracting $E$ gives us $A^2=0$, so that $A$ is a fiber of type  $I_1$ or $II$. In case 2), $B$ is a $(-2)$-curve passing through the singularity of $A$, which is impossible.
 In case 3)  if $r\ge 4$  the curve $A$ intersects a $(-2)$ curve, which is not possible. 

Finally, for $r-d=0$ one has $n=2$ and we have listed all the possibilities.
\end{proof} 
\begin{rem}\label{rem: -3curve n = 3}
 Note that we have shown in the course of the proof that in the case $n=3$ the minimal elliptic surface $Y$ contains a $(-3)$-curve. 
\end{rem}

\begin{lem} \label{lem: Y=Y}
If $X$ has index $n=2$, then:
\begin{enumerate} 
\item the surface $\tilde Y=Y$ is minimal;
\item the canonical system $|K_Y|$ is equal to the  pencil $|F|$, where $F$ is an elliptic fibre. 
\end{enumerate}
\end{lem}
\begin{proof}
(i) This was part of Lemma \ref{lem: Y-ell}. 
\medskip

(ii) 
Since $p_g(Y)=2$, the canonical bundle formula for elliptic surfaces gives  $|K_Y|=|F| +\sum (m_i-1)F_i$, where  $m_1F_1, 
\dots m_kF_k$ are  the multiple fibers. 
The exceptional divisor of the desingularization map $Y\to X$ contains a $(-n)$-curve $B$ for $n=3$ or $4$. So we have $1\le K_YB=B(F +\sum (m_i-1)F_i)\le n-2\le 2$, and we conclude that   there are no multiple fibers and  $|K_Y|=|F|$.
\end{proof}
\begin{prop} \label{prop: index 2 classification}
If $X$ has index $n=2$, then it is obtained as in Example \ref{ex: n=2}. 
\end{prop}
\begin{proof} 
Denote by $\Delta$ the exceptional divisor of $Y\to X$; recall that $\Delta$   is a string of type $[4]$, $[3,3]$, or $[3,2\dots,3]$ 
(with $2$ occurring $k$ times), according to whether $d=1,2$ or $d=k+2>2$.   Since $(K_Y+\frac 12 \Delta) \Delta =0$  and  $K_Y=F$ (cf. Lemma \ref{lem: Y=Y}), we have $F\Delta=2$. 

Set $L:=3F+\Delta$; one has $L^2=8$, $LK_Y=2$ and so $\chi(L)=6$. We can write $L=K_Y+(2F+\Delta)$, and $2F+\Delta$ is nef and big since it is the pull-back of $2K_X$, so 
 by Kawamata-Viehweg vanishing  $h^0(L)=\chi(L)=6$.
 Restricting to $\Delta$ and taking cohomology we see that the image of the  map $H^0(3F+\Delta)\to H^0(\OO_{\Delta}(2))$ has dimension 2. 
 We are going to  use this fact to show that  $|L|$ has no fixed components. Note that any fixed component of $L$ must be a component of $\Delta$, since $|3F|$ is free. 
  If $\Delta$ is irreducible, then it is not a fixed component since the map $H^0(L)\to H^0(L|_{\Delta})$ is non-zero. 
If $\Delta$ is reducible, denote by $A_1$ and $A_2$ the $(-3)$-curves of $\Delta$ and suppose  that  $A_1$  is in the base locus of $|L|$: then all the $(-2)$-curves of $\Delta$ are also in the base locus of $|L|$, 
since $L\Gamma=0$ for a $(-2)$-curve $\Gamma$ of $\Delta$, and $|L|$ has a base point on $A_2$, so $|L|_{\Delta}$ has dimension $\le 0$, a contradiction. 
 Assume now that $d>2$ and  a $(-2)$-curve of  $\Delta$ is a fixed component of $|L|$: then, as above,  all the $(-2)$-curves of $\Delta$ are fixed components of $|L|$ and
  $h^0(3F+A_1+A_2)=h^0(L)=6$.  Set $M=3F+A_1+A_2$: we have $\chi(M)=5$, so $h^1(M)=1$. On the other hand, $M=K_Y+2F+A_1+A_2$ and we can  write
   $2F+A_1+A_2=(2F+\frac 23 A_1+\frac 23 A_2)+\frac 1 3A_1+\frac 13 A_2$. Since $2F+\frac 23 A_1+\frac 23 A_2$ is nef and big,  we have    $h^1(M)=0$ by Kawamata-Viehweg's vanishing, 
   a contradiction. We conclude  that $|L|$ has no fixed component.
\newline
Now one can argue precisely as in the proof of Lemma \ref{lem: elliptic-3} and prove the following:
\begin{itemize}
\item $|L|$ is base point free and defines a 2-to-1 map $\fie\colon Y\to \pp^5$, so the image $\Sigma$ of $\fie$ is either a smooth rational scroll of degree $4$ or a cone over the rational normal curve of degree 4;
\item $\fie$ maps the elliptic fibers  2-to-1 to  rulings of $\Sigma$ and  maps $\Delta$ to a line meeting all the rulings, so $\Sigma$ is isomorphic to $\F_2$; 
\item the curves contracted by $\fie$ are exactly the $(-2)$-curves contained in $\Delta$ (if any). 
\end{itemize}
So $X$ is obtained as in Example \ref{ex: n=2}. 
\end{proof}

\begin{proof}[Proof of Theorem \ref{thm: oneT}]
We have restricted the number of possible singularities in Lemma \ref{lem: Y-ell}. The fact that the cases of index $n=2$ are constructed as in Example \ref{ex: n=2} is proved in Proposition \ref{prop: index 2 classification} which gives the bound $d\leq 32$ from Proposition \ref{prop: D}. 

It remains to show that if $X$ has a singularity of type $\frac{1}{18}(1,5)$, then $X$ is constructed as in Example \ref{ex: (b)}, and if $X$ has a singularity of type $\frac{1}{25}(1,14)$, then $X$ is constructed as in Example \ref{ex: (c)}. As explained in  Lemma \ref{lem: Y-ell} the corresponding T-strings in $\tilde Y$ contain a $(-3)$-curve and we claim that this maps indeed to a $(-3)$-curve $B$ in $Y$. In the former case, this follows from  Remark \ref{rem: -3curve n = 3}, while in the latter case  this 
is~\cite[Theorem 3.2 (A1)]{rana-urzua17}.
We know explicitly the possible such  pairs $(Y,B)$  by Lemma \ref{lem: elliptic-3} and thus 
the T-singularities arise as in the Examples  \ref{ex: (b)} and \ref{ex: (c)}. 
%
%
\end{proof}
\begin{proof}[Proof of Corollary  \ref{cor: moduli}]
Recall that the main component of the moduli space of I-surfaces is irreducible of dimension $28$.
We have shown in Remark \ref{rem: bic-n=2} that every T-singular surface of index 2 is smoothable and in Proposition \ref{prop: D} that those of type $\frac 14(1,1)$ depend on $27$ parameters, hence give a divisor in the main component. For type $\frac 1{18}(1,5)$ we argue similarly using  Corollary \ref{cor: index 3 smoothable} and Remark \ref{rem: moduli-n=3}.
The case of type $\frac{1}{25}(1,14)$ was treated in Proposition \ref{prop: 5 I_1 not smoothable}.
\end{proof}

  \def\cprime{$'$}

 \end{document}